\documentclass[a4paper]{article}

\usepackage{amssymb}
\usepackage{amsmath}
\usepackage{amsthm}
\usepackage{amsfonts}
\usepackage{graphicx}
\usepackage{subfigure}
\usepackage{enumerate}
\usepackage{cmap}

\theoremstyle{plain}
\newtheorem{nn}{}[section]
\newtheorem{theorem}[nn]{Theorem}
\newtheorem{lemma}[nn]{Lemma}

\newtheorem{proposition}[nn]{Proposition}
\newtheorem{remark}[nn]{Remark}

\theoremstyle{definition}
\newtheorem{definition}[nn]{Definition}

\newcommand{\R}{\mathbb{R}}

\makeatletter

\renewcommand{\@listI}{%
\leftmargin=40pt
\rightmargin=0pt
\labelsep=5pt
\labelwidth=20pt
\itemindent=0pt
\listparindent=0pt
\topsep=2pt plus 1pt minus 1pt
\partopsep=1pt plus 1pt
\parsep=1pt plus 1pt
\itemsep=\parsep}

\renewcommand{\@listii}{%
\leftmargin=25pt
\rightmargin=0pt
\labelsep=5pt
\labelwidth=20pt
\itemindent=0pt
\listparindent=0pt
\topsep=0pt plus 1pt
\partopsep=0pt plus 1pt
\parsep=0pt
\itemsep=\parsep}

\makeatother

\title{Some possible numbers of edge coverings of a bipartite graph or shortest paths with fixed ends in a space of compact sets in $R^n$}
\date{}
\begin{document}
\maketitle
\section{Introduction}

Hausdorff distance was introduced in the beginning of XX century in order to measure distance between compact sets. This function behaves as a distance function on a set of compacts of an arbitrary metric space $X$, denoted as $\mathcal{H}(X)$.

In \cite{Core} numbers of shortest paths between points in $\mathcal{H}(\R^n)$ were studied, a link with a graph theory was found and was shown that for any number from 1 to 36 (except 19) there is a pair of compact sets in $\R^n$ such that there is such a number of shortest paths between them. It was shown that there is no such pair for 19.

This work expands results given in \cite{Core}. By using machine computation it will be shown that there is no such pair of compact sets in $\R^n$ that there are 19, 37, 41, 59 or 67 shortest paths in $\mathcal{H}(\R^n)$ between this two compacts, and there can be found such pairs for all natural numbers from 1 to 1000 with the exception of 19, 37, 41, 59, 67, 82, 97, 149, 197, 223, 257, 291, 379 (for the last 8 numbers we can not give a definite answer).

The author would like to thank Prof. A.O. Ivanov and Prof. A.A. Tuzhilin for advice and discussion regarding the problem. The work is supported by Russian Foundation for Basic Research, Project 13–01–00664, the Russian Government project 11.G34.31.0053, and Russian President Grant for Leading Scientific Schools Support, Project NSh–1410.2012.1.

\section{Definitions and previous results}

It is assumed implicitly everywhere in this work that no graph can have multiple edges.

\begin{definition}
Let $(X, \rho)$ be a metric space, $a \in X$ is a point in this space and $B \subset X$ is an arbitrary set. Then \emph{distance between point and this set} is the following number: $\rho(a, B) := \inf\limits_{b \in B}\rho(a, b)$.
\end{definition}

\begin{definition}
\emph{Hausdorff distance} between two compact sets $A, B$ in metric space $X$ is a function $\rho_H(A, B) = \max(\sup\limits_{a \in A}\rho(a, B), \sup\limits_{b \in B}\rho(b, A))$.
\end{definition}

\begin{definition}
A \emph{parametrized curve} is a continuous function $f \colon [0,1] \rightarrow X$. Points $f(0)$ and $f(1)$ are called \emph{endpoints}.
\end{definition}

\begin{definition}
Two parametrized curves $f$ and $g$ are equivalent if there is monotonic function $h \colon [0,1] \rightarrow [0,1]$ such that $h(0) = 0, h(1) = 1$ and $f = g \circ h$.
\end{definition}

\begin{definition}
A class of equivalence of parametrized curves is called a curve.
\end{definition}

\begin{definition}
Let $\gamma$ be a curve and $f \colon [0,1] \rightarrow X$ is one of parametrized curves in $\gamma$. Then \emph{length} of a curve is a $\ell(\gamma) = \sup\limits_n \sup\limits_{\{0 = t_0 < t_1 < \ldots < t_n = 1\}}\sum\limits_{i = 0}^{n-1}\rho(f(t_i), f(t_{i+1}))$.
\end{definition}

\begin{definition}
A curve is called \emph{shortest} if any other curve between the same endpoints has more or equal length.
\end{definition}

\begin{definition}
\emph{An edge covering} of a graph $G = (V, E)$ is such a subset $E'$ of its edges that any vertex of graph is incident to at least one edge from $E'$.
\end{definition}

In \cite{Prep} it was shown that a problem of finding a number of shortest paths between points in $\mathcal{H}(\R^n)$ can be reduced to the problem of finding a number of edge coverings of a bipartite graph, namely:

\begin{theorem}
For any pair of points $(A, B)$ in $\mathcal{H}(\R^n)$ such a bipartite graph $G$ can be constructed, that its number of edge coverings will be equal to number of shortest paths between points $A$ and $B$. For any bipartite graph $G$ there exists such a pair of compacts in $\R^n$ that number of edge coverings of $G$ equals number of shortest paths between these two compacts in $\mathcal{H}(\R^n)$.
\end{theorem}

Further we will discuss this modified question: how many edge coverings a bipartite graph can have?

\section{Graph decomposition}

In this chapter and further all graphs will be assumed to be connected if not explicitly stated otherwise (for a graph with several connected subgraphs its number of edge coverings equals to the product of numbers of coverings of its connected componenet). Furthermore, we will assume that if we remove vertex from a graph, we also remove all it's incident edges.

We will denote number of edge coverings of graph $G$ as $\alpha(G)$.

\begin{definition}
\emph{A decomposition} of a connected graph $G = (V, E)$ is a finite set of non-empty connected subgraphs $G_i = (V_i, E_i), i = 1,\ldots,n, n 
ge 2$ such that:

1. Each subgraph $G_i$ contains at least one edge.

2. Any two different subgraphs $G_i$, $G_j$ have no common edges and no more than one common vertex.

3. $\bigcup\limits_{i=1}^n V_i = V$, $\bigcup\limits_{i=1}^n E_i = E$.

4. If we construct a graph $H$ with its vertices being graphs $G_i$ and there is an edge between $G_i$ and $G_j$ if and only if $V_i \cap V_j \ne \emptyset$, then this graph is a connected tree. Graph $H$ is called \emph{a graph of decomposition}.
\end{definition}

\begin{definition}
Graph $G$ with at least one edge such that it does not have a decomposition is called \emph{atomic}.
\end{definition}

\begin{remark}
Graphs in a decomposition don't have to be atomic. Moreover, there may be no decomposition on only atomic graphs. For example, a \"star\" with tree rays can be split only on a star with two rays and an edge.
\end{remark}

\begin{remark}
Any non-atomic graph can be decomposed using exactly two subgraphs.
\end{remark}

\begin{lemma}
\label{ZeroNewPoints}
Let $G = (V, E)$ be a connected graph, $e = (uv) \notin E$ --- an edge between vertices $u, v \in V$ and $G' = (V, E \cup \{e\})$, then $\alpha(G') \ge 2 \alpha(G)$.
\end{lemma}
\begin{proof}
For any edge covering $E'$ of graph $G$ there is at least two edge coverings of $G'$: $E'$ and $E' \cup \{e\}$.
\end{proof}

Next two lemmas can be easily proven in the same way.

\begin{lemma}
\label{OneNewPoint}
Let $G = (V, E)$ be a connected graph, $v \notin V$ --- vertex not in the graph, $a, b \in V$ --- vertices that belong to graph $G$ and $G' = (V \cup \{v\}, E \cup \{(av), (vb)\})$. Then $\alpha(G') \ge 3 \alpha(G)$.
\end{lemma}

\begin{lemma}
\label{ManyNewPoints}
Let $G = (V, E)$ be a connected graph, $v_1, \ldots, v_n \notin V$  --- vertices not in the graph, $n \ge 2$, $a, b \in V$ --- vertices that belong to graph $G$ and $G' = (V \cup \{v_1, \ldots, v_n\}, E \cup \{(av_1), (v_1, v_2), \ldots, (v_{n-1}, v_n)(v_nb)\})$. Then $\alpha(G') \ge 5 \alpha(G)$.
\end{lemma}

We will need an algorithm to count a number of edge coverings of an arbitrary graph. This algorith can be presented using the next proposition.

\begin{proposition}
Let $G = (V, E)$ be a connected graph, and $e = (uv) \in E$ --- and edge in this graph that connects vertices $u$ and $v$. Then $\alpha(G) = 
2 \alpha(G \setminus e) + \alpha(G \setminus u) + \alpha(G \setminus u) + \alpha(G \setminus \{u, v\})$.
\end{proposition}
\begin{proof}
Let $E'$ be an arbitrary edge coverings. There is five options available:

1) There are edges in $E' \setminus e$ that are incident to $u$ and $v$ and $e \notin E'$. There is exactly $\alpha(G \setminus e)$ such coverings.

2) There are edges in $E' \setminus e$ that are incident to $u$ and $v$ and $e \in E'$. There is exactly $\alpha(G \setminus e)$ such coverings.

3) There is at least one  edge in $E' \setminus e$ that is incident $u$ but no edges incident to $v$. Then $e \in E'$ and there is exactly $\alpha(G \setminus v)$ such coverings.

4) There is at least one  edge in $E' \setminus e$ that is incident $v$ but no edges incident to $u$. Then $e \in E'$ and there is exactly $\alpha(G \setminus u)$ such coverings.

5) There is no edge in $E' \setminus e$ that is incident to $u$ or $v$. Then $e \in E'$ and there is exactly $\alpha(G \setminus e)$ such coverings.

If we add all options we will get the equation in the proposition.
\end{proof}

Now we can construct inductive algorithm to count number of edge coverings in an arbitary graph: take an arbitrary edge of this graph and, using the formula above, count number of edges of appropriate subgraphs and then add them (the process will stop for any finite graph, because one inductive step further there is at least one edge less). If graph has no edges then if it has no vertices either there is one edge covering (empty), and if it has at least one vertex, there is no edge coverings of such graph.

This algorithm is slow and not optimal. Luckily, we won't be using it much.

\begin{proposition}
\label{SevenGraphs}
There is exactly 7 different atomic bipartite graphs with number of edge coverings no more than 67.
\end{proposition}
\begin{proof}
Graph with two vertices and one edge between them is atomic. Obviously, any other connected graph with a vertex of degree one is not atomic. It can be split on a graph, consisting from this vertex and it's incident edge and the rest of the graph.

Let $G = (V, E)$ be a connected bipartite atomic graph with more than one edge. Each vertex of this graph has degree of at least two. It means that such graph contains a cycle. Consider the cycle of maximum length in this graph. Since $G$ is bipartite, any cycle has even number of edges.

Let $X$ be a connected subgraph of graph $G$ with at least one edge. If $G$ has vertices which do not belong to $X$, then there is at least one edge from vertex $u$ in $X$ to vertex in $G \setminus X$. Let's denote vertex in $G \setminus X$ as $v$ and edge as $e = (uv)$. Let $V'$ be a set of vertices in $G$ such that for any vertex from $V'$ there exists a path to vertex $v$ which does not contain vertex $u$ and edge $e$. All of the vertices in $X$ either in $V'$ together or not in $V'$ together. If the y are not, graph $G$ can be split in two subgraphs connected only by a vertex $u$ which contradicts the assumption that $G$ is atomic. Then there is a path from vertex $v$ to vertex from $X$ which does not contain edge $e$ and vertex $u$.

Assume that $G$ has a cycle of 10 or more edges. As was proven in \cite{Old}, such cycle has at least 123 edge coverings all by itself and because of \cite[т.~6.1]{Core} graph $G$ has no less coverings which contradicts assumption that $\alpha(G) \le 67$. Then there is no cycle of 10 or more edges.

Let the maximum number of edges in a cycle be 8. This cycle has 47 edge coverings. Since $47 * 2 > 67$, then using \ref{ZeroNewPoints} $G$ has no edges between vertices of this cycle. If there is no other vertices then $G$ is a cycle of length 8 (see p.2). If there are other vertices, then using previous statement with this cycle as an $X$ we can see that there is a path with at least two edges between vertices of a cycle which by \ref{OneNewPoint} and \ref{ManyNewPoints} gives us that $\alpha(G) \ge 47 *3 > 67$. Then there is no additional vertices.

Let the maximum number of edges in a cycle be 4, this cycle has 7 edge coverings. Then for any $X$ additional path can have no more than two edges. Let us take cycle as $X$. There can be no additional edges between vertices of cycle, because graph $G$ is bipartite. If there are no vertices in $G \setminus X$ then $G$ is a cycle of 4 edges (see p.3). If there is at least one additional vertex, then two edges can be added to $X$. Cycle with two additional edges can look only like one on the p.4, it has 25 edge coverings. Let's take this subgraph as $X$. If there is no other vertices then $G = X$. If there is a vertex in $G \setminus X$ then there is two edges that can be added to $X$ and by \ref{OneNewPoint} $\alpha(G) \ge 25 * 3 > 67$. Then there is no additional vertices.

Let the maximum number of edges in a cycle be 6, this cycle has 18 edge coverings. If $G$ has no additional vertices, then it can have no more that one additional edge ($18 * 4 > 67 > 18 * 2$). If there is no additional edges, then $G$ is a cycle of 6 edges (see p.5). If there is additional edge, it can be placed only as it is done on p.6 since $G$ is bipartite.

If graph $G$ has additional vertices, then there is only one such vertex ($18 * 3 < 67 < 18 * 5< 18 * 3 *3$). One additional vertex can be placed only as shown on p.7, this graph has 66 edge coverings which means there are no other edges.
\end{proof}

\section{Precoverings and formulas for counting their number.}

We will select one vertex in a graph and call it \emph{topmost} or \emph{selected} and  denote it in the following way:

$$\mathcal{G} = (G, v),\qquad G = (V, E),\qquad v \in V.$$

\begin{definition}
Let $\mathcal{G} = (G, v)$ be a connected graph with selected vertex. We will define four functions:

$\alpha(\mathcal{G}) = \alpha(G),$

$\beta(\mathcal{G}) = \alpha(G'),$

где $G' = (V', E')$, а $V' = V \setminus \{v\}, E' = E \setminus \{(uv)| u \in V\}$
$s(\mathcal{G}) = \alpha(\mathcal{G}) + \beta(\mathcal{G})$,

$f(\mathcal{G}, X) = \begin{cases}
s(\mathcal{G}),&\text{if $v \in X$;}\\
\alpha(\mathcal{G}),&\text{if $v \notin X$.}
\end{cases}
$

\end{definition}

It is useful to think that for graph $\mathcal{G}_0 = (\{v\},v)$ with only one vertex this functions are equal to $\alpha(\mathcal{G}_0) = 0$ and $\beta(\mathcal{G}_0) = 1$ respectively.

\begin{definition}
\emph{An edge precovering} of graph $G$ with selected vertex $v$ is such a subset $E'$ of its vertices that any vertex in the graph $G$, except $v$ is incident to at least one edge from $E'$ (vertex $v$ may or may not be incident to edges from $E'$).
\end{definition}

\begin{proposition}
\label{HalfCover}
Let $\mathcal{G} = (G, v)$ be a connected graph with selected vertex. Then it has exactly $s(\mathcal{G})$ edge precoverings.
\end{proposition}
\begin{proof}
Let $E'$ be a precovering of $\mathcal{G}$. If $E'$ is not a covering of $G$ then $E'$ is a covering of $G' = (V \setminus \{v\}, E \setminus \{(uv)| u \in V\})$. There is exactly $\beta(\mathcal{G})$ such coverings. If $E'$ is a covering of $G$ then it is a precovering of $\mathcal{G}$ and there is $\alpha(\mathcal{G})$ such coverings.
\end{proof}

\begin{proposition}
Let $\mathcal{G}_i = (G_i, v), i = 1, \ldots, n$ be several graphs sharing a common selected vertex and having no other vertices, common for any two graphs. Then $\mathcal{G} = (\bigcup\limits_{i=1}^n G_i, v)$ has the following properties:

1) $\beta(\mathcal{G}) = \prod\limits_{i=1}^n \beta(\mathcal{G}_i),$

2) $\alpha(\mathcal{G}) = \prod\limits_{i=1}^n  s(\mathcal{G}_i) - \prod\limits_{i=1}^n \beta(\mathcal{G}_i).$

\end{proposition}
\begin{proof}
1) If we remove vertex $v$ from $\bigcup\limits_{i=1}^n G_i$ and edges, incident to it, graph $G$ will fall apart on $n$ connected graphs with $\beta(\mathcal{G}_1), \ldots, \beta(\mathcal{G}_n)$ each. Number of edge coverings of their union equals product of numbers of their coverings, i.e. $\prod\limits_{i=1}^n \beta(\mathcal{G}_i)$.

2) Let $E'$ be a precovering of $\bigcup\limits_{i=1}^n G_i$ and $E'_i = E' \cap G_i$. Clearly, $\bigcup\limits_{i=1}^n E'_i = E'$ and $E'_i$ is a precovering of graph $\mathcal{G}_i$ with selected vertex. So, each precovering defines and is uniquely defined by a set of $n$ nonintersecting precovering. Then number of precoverings of graph $\mathcal{G}$ can be count as  $s(\mathcal{G}) = \prod\limits_{i=1}^n  s(\mathcal{G}_i)$ and number of edge coverings --- as $\alpha(\mathcal{G}) = s(\mathcal{G}) - \beta(\mathcal{G}) = \prod\limits_{i=1}^n  s(\mathcal{G}_i) - \prod\limits_{i=1}^n \beta(\mathcal{G}_i)$.
\end{proof}

\begin{definition}
Let $G = (V, E)$ be a connected graph, $V = \{v_1, v_2, \ldots, v_n\}$, $V'$ be a subset of $V$ (possibly, empty), then graph $G(V') = (V', E')$, where $E'$ is a set of all edges in $E$ such that for both of it's vertices are in $V'$.
\end{definition}

\begin{proposition}
Let $\mathcal{G}_0 = (G_0, v)$ be a connected graph with selected vertex. Let $E_0$ be a set of it's edges and $V_0 = \{v, v_1, \ldots, v_n\}$ be a set of it's vertices. Also let $\mathcal{G}_i = (G_i, v_i)$ be graphs with selected vertices such that different graphs $G_i$ and $G_j$ have no common vertices or edges and $G_i \cap G_0 = \{v_i\}$. Then $\mathcal{G} = (\bigcup\limits_{i=0}^n G_i, v)$ has the following properties:

1) $\alpha(\mathcal{G}) = \sum\limits_{V' \subset V_0 | v \in V'} \alpha(G_0(V')) \prod\limits_{i=1}^n f(\mathcal{G}_i, V')$, 

2) $\beta(\mathcal{G}) = \sum\limits_{V' \subset V_0 | v \notin V'} \alpha(G_0(V')) \prod\limits_{i=1}^n f(\mathcal{G}_i, V')$.

\end{proposition}
\begin{proof}

1) Fix $V' \subset V_0$. Let $E'$ be an edge covering of $G = \bigcup\limits_{i=0}^n G_i$ and $E'_i = E' \cap G_i$ such that for any vertex $u \in V_0$ holds the following: $u \in V'$ if and only if there is an edge in $E'_0$ incident to $u$. Apparently, $v \in V'$, because all edges, incident $v$ in graph $G$, belong to $E_0$.

Let us count a number of edge coverings for this fixed $V'$. Since $G_i$ and $G_0$ intersect only by a vertex, and $G_i$ and $G_j$ don't intersect at all, then all possible $E'_i$, $i \ne 0$ are precoverings. Then any edge covering $E'$ uniquely defines and is defined by a set if precoverings $E'_i$ and set $E'_0$.

If $v_i \notin V'$, then $E'_i$ hs to be a covering and any covering will do. There is $\alpha(\mathcal{G}_i)$ such coverings. If $v_i \in V'$ then $E'_i$ can be any precovering of $\mathcal{G}_i$, there is $s(\mathcal{G}_i)$ of them \ref{HalfCover}. Then there exist $f(\mathcal{G}_i, V')$ possible values of $E'_i$.

Set $E'_0$ can be any set of edges such that for any vertex $u \in V'$ there is an edge in $E'_0$ incident to it and for any vertex $w \notin V'$ there are none. Then any suitable set od edges $E'_0$ is a covering of $G(V')_0$, there is $\alpha(G_0(V'))$ such coverings.

Then there is $\alpha(G_0(V')) \prod\limits_{i=1}^n f(\mathcal{G}_i, V')$ edge coverings such that abovementioned condition on $V'$ works. By adding all possible $V'$ we get formula in the proposition.

2) The proof is exactly as in case 1 with the exception, that $v$ don't belong $V'$, because there is no edges, incident to $v$ in a covering of $G \setminus \{v\}$.
\end{proof}

\section{Algorithm and calculations.}

Two previous statements from previous chapter describe two ways of getting new graphs with selected vertex from given graphs with selected vertices. First, one can \"glue\" several graphs using selected vertex. Second, one can glue several graphs with selected vertex to a given atomic graph with selected vertex, gluing selected vertices in graphs to different vertices of atomic graph (except selected vertex of atomic graph, that stays unglued to anything).

We will prove that using this tow ways one can obtains all connected graphs with selected vertices wit hno more than given number of edge coverings.

\begin{remark}
Let $G$ be non-atomic graph and $G_1, \ldots, G_n$ is it's decomposition. Then vertex $v \in G$ may belong either one of graphs $G_1, \ldots, G_n$, or two of them.
\end{remark}

\begin{lemma}
\label{DeeperLevel}
Let $G = (V,E)$ be a connected non-atomic graph and $G_1, \ldots, G_n$ is a decomposition of this graph, $K_1, \ldots, K_k$ is a decomposition of graph $G_1$, $k \ge 2$. Let vertex $v \in G_1$ belong to no graph of $G_2, \ldots, G_n$ (it belongs only to $G_1$). Then if $v$ belongs exactly one pf graphs $K_1, \ldots, K_k$, namely $K_1$, then there exists decomposition of $G$ with subgraphs $K_1, L_1, \ldots, L_m$ such that $v \notin L_i$. If vertex $v$ belongs two graphs from $K_1, \ldots, K_n$, then there exists decomposition of $G$ with subgraphs $L_1, L_2$ such that $v \in L_i$.
\end{lemma}
\begin{proof}

Let's consider a set $K = \{K_1, \ldots, K_k, G_2, \ldots, G_n\}$. The first three conditions, required to call it a decomposition, hold. It can be not a decomposition only if at least three graphs from $K$ share a common vertex which gives us a cycle (there can be no more that three graphs, since no more than two graphs in $K_1, \ldots, K_k$ and no more than two graphs in $G_1, \ldots, G_n$ can share a vertex, then only a triple $K_a, K_b, G_c$  can share a vertex). Then from graphs $K_a$ and $K_b$ no more than one can have a vertex $v$ since two graphs $K_a$, $K_b$ cannot share more than one vertex (without loss of generality we can assume that $v \notin K_b$). Let's replace sets $K_b$ and $G_c$ with set $K_b \cup G_c$ in set $K$. Apparently, for these new graphs the first three conditions will hold, and they will not generate new cycles. Renewed set will be a decomposition of graph $G$ and graphs from decomposition of $G_1$ that contain vertex $v$ will not be changed. There can be two options:

1) vertex $v$ belongs only graph $K_1$ then we got the required decomposition,

2) vertex $v$ belongs graphs $K_1$ and $K_2$, then in decomposition graph of $K$ there is an edge between these two graphs. Since decomposition graph is a tree, set $K$ is split by this edge on two sets $M_1$ and $M_2$, union of all elements of set $M_i$ we will denote as $L_1$ and $L_2$ respectively. Clearly, $L_1 \cap L_2 = \{v\}$, then $L_1, L_2$ is the required decomposition of graph $G$.
\end{proof}

\begin{lemma}
\label{TwoOptions}
Let vertex $v$ belong non-atomic graph $G$, then one of two holds:

1) there exisits a decomposition of graph $G$ with subgraphs $L_1, L_2$ such that $v \in L_i$,

2) there exists a decomposition of graph $G$ with subgraphs $K, L_1, \ldots, L_n$ such that $K$ is an atomic graph, $v \in K$, $v \notin L_i$ and $L_i \cap K \ne \varnothing$ (decomposition graph is a star with its center in $K$).
\end{lemma}
\begin{proof}
Take an arbitrary decomposition of $G$. We will apply \ref{DeeperLevel} to this decomposition util one of two will hold:

1) decomposition of $G$ has two graphs that contain vertex $v$ then using \ref{DeeperLevel}, construct a decomposition $L_1, L_2$ of graph $G$ such that $v \in L_i$,

2) decomposition of $G$ has one graph that contains vertex $v$ and this graph is atomic (on of two conditions will eventually become true, since inc case  1 of lemma \ref{DeeperLevel} graph of decomposition $K_1$ that contains vertex $v$ has less edges than initial graph $G_1$). We will denote atomic graph as $K$ and decomposition as $K, G_1, \ldots, G_m$. In decomposition graph vertex $K$ splits graph on several connected components $A_1, \ldots, A_n$. Union of graphs of a connected component we will denote as $L_1, \ldots, L_n$ respectively. Since $v \notin L_i$ и $L_i \cap K \ne \varnothing$ this is the required decomposition.
\end{proof}

\begin{proposition}
Using the two ways to construct graphs that were described above (the second is used with seven graphs from \ref{SevenGraphs}) we can construct any graph with selected vertex with a number of edge coverings no more that 67.
\end{proposition}
\begin{proof}
We will prove this proposition by induction by number of edges $\#E = n$.

If $n=1$ there exists only one connected graph with one edge, it is atomic and is among the seven graphs we use.

Let's assume that we can construct any connected graph with less that $n$ edges and no more than 67 edge coverings. Let $\mathcal{G} = (G, v), G = (V, E)$ be an arbitrary connected graph with $n$ edges and no more than 67 edge coverings. If it is atomic it is among  the seven graphs we use and it can be constructed with the use of second way. If it is non-atomic then by \ref{TwoOptions} there can be two options:

1) there is a decomposition of $G$ with subgraphs $L_1, L_2$ such that $v \in L_i$. Then we can construct graphs $L_1$ and $L_2$ with selected vertex $v$ (we can do this by assumption) and connect them using way 1,

2) there is a decomposition of $G$ with subgraphs $K, L_1, \ldots, L_n$ such that $K$ is an atomic graph, $v \in K$, $v \notin L_i$ and $L_i \cap K \ne \varnothing$. Then we can construct graphs $L_1, \ldots, L_n$ with selected vertices and combine them with atomic graph $K$ using way 2.
\end{proof}

The algorithm for constructing graphs is the following: we have a pool of pairs $(\alpha, \beta)$. Initially there is only one pair $(0, 1)$ which is the pair for empty set. At each step we try to construct all pairs we can using ways one and two and constructed graphs in the pool. We add all pairs that weren't encountered before to the pool. If after any step we didn't add a pair with $\alpha < x$ then we won't add such a pair at all and no graph with $x$ edge coverings can be constructed from given atomic graphs. We can keep running the algorithm to find what  \emph{can} be constructed, though.

We have implemented this algorithm is C++ and C\# programming languages and this programs gave us identical results:

\begin{proposition}
There is no bipartite graph such that it has 19, 37, 41, 59 or 67 edge coverings. A bipartite graph can have any number of edge coverings from 1 to 1000, except 19, 37, 41, 59, 67, 82, 97, 149, 197, 223, 257, 291 and 379.
\end{proposition}

\section{Unsolved problems and possible applications of an algorithm}

This algorithm turned out to be pretty inefficient and needs some improvements. We believe this can be done and it can check for us numbers 82 and 97 (we need only one more atomic graph). However, for the algorithm to work with larger numbers it needs exponentially growing number of atomic graphs which makes it totally useless for 149 and above.

As one can see, this algorithm can be used not only with atomic graphs, but with any graph, with the same problems in exponentially growing number of atomic graphs. This algorithm can also be used to find possible edge coverings in a trees. This requires only one atomic graph and is pretty fast. For example, we were able to obtain this result in less than a minute of computer time: from 1 to 256 we can not construct trees with 19, 37, 41, 57, 59, 67, 79, 82, 97, 111, 131, 149, 177, 179, 197, 201, 205, 223, 237 and 251, others can be constructed.

All the atomic graphs were planar so far, and the structure generated by them also planar, but this changes if we take larger atomic graphs. We saw no new pairs when we added non-planar graphs to the algorithm. The question is the following: is there a non-planar graph that there is no planar graph with the same number of edge coverings?

\end{document}